\documentclass[oneside,a4paper]{article}
\usepackage{amssymb,amsmath,latexsym,graphicx,tikz,indentfirst,scalerel,hyperref,caption,marginnote,multicol,cite}

\hypersetup{
  colorlinks   = true, 
  urlcolor     = blue, 
  linkcolor    = blue, 
  citecolor   = red 
}

\usetikzlibrary{automata, positioning, calc,arrows}

\def\dj{d\kern-0.4em\char"16\kern-0.1em}
\def\Dj{\mbox{\raise0.3ex\hbox{-}\kern-0.4em D}}

\newtheorem{theorem}{Theorem}[section]
\newtheorem{lemma}[theorem]{Lemma}
\newtheorem{proposition}[theorem]{Proposition}
\newtheorem{corollary}[theorem]{Corollary}
\newtheorem{definition}[theorem]{Definition}

\makeatletter
\renewcommand{\@dotsep}{10000}
\makeatother

\newenvironment{proof}
{\noindent
{\it Proof.}}
{\hspace{\stretch{1}}%
$\Box$}

\newcounter{primer}[section]

\DeclareMathOperator{\diam}{diam}

\DeclareMathOperator{\Cen}{Cen}

\makeatletter 
\tikzset{my loop/.style =  {to path={
  \pgfextra{}
  [looseness=4,min distance=5mm]
  \tikz@to@curve@path},font=\sffamily\small
  }}  
\makeatletter

\newcommand{\rightarrowpm}{\stackrel{\pm}{\rightarrow}}
\newcommand{\leftarrowpm}{\stackrel{\pm}{\leftarrow}}
\newcommand{\simp}{\stackrel{p}{\sim}}

\newcommand{\simppm}{\stackrel{p_\pm}{\sim}}

\makeatletter
\newcommand{\myitem}[1]{%
\item[#1]\protected@edef\@currentlabel{#1}%
}
\makeatother

\begin{document}

\thispagestyle{empty}
\begin{center}
\Large{The Power Graph of a Torsion-Free Group Determines the Directed Power Graph}
\vspace{3mm}

\small{Samir Zahirovi\' c\\
{\it Department of Mathematics and Informatics, Faculty of Sciences,\\
University of Novi Sad, Serbia\\
\href{mailto:samir.zahirovic@dmi.uns.ac.rs}{samir.zahirovic@dmi.uns.ac.rs}}}
\end{center}

\begin{abstract}
The directed power graph $\vec{\mathcal G}(\mathbf G)$ of a group $\mathbf G$ is the simple digraph with vertex set $G$ such that $x\rightarrow y$ if $y$ is a power of $x$. The power graph of $\mathbf G$, denoted by $\mathcal G(\mathbf G)$, is the underlying simple graph.

In this paper, for groups $\mathbf G$ and $\mathbf H$, the following is proved. If $\mathbf G$ has no quasicyclic subgroup $\mathbf C_{p^\infty}$ which has trivial intersection with every cyclic subgroup $\mathbf K$ of $\mathbf G$ such that $\mathbf K\not\leq\mathbf C_{p^\infty}$, then $\mathcal G(\mathbf G)\cong \mathcal G(\mathbf H)$ implies  $\vec{\mathcal G}(\mathbf G)\cong \vec{\mathcal G}(\mathbf H)$. Consequently, any two torsion-free groups having isomorphic power graphs have isomorphic directed power graphs.
\end{abstract}

\section{Introduction}

The directed power graph of a group, which was introduced by Kelarev and Quinn \cite{prvi kelarev i kvin}, is the simple directed graph whose vertices are all elements of the group, and in which $x\rightarrow y$ if $y$ is a power of $x$, i.e.  $y=x^n$ for some $n\in\mathbb Z$. The power graph of a group is the underlying simple graph, and it was first studied by Chakrabarty, Ghosh and Sen \cite{cakrabarti}. The power graph has been the subject of many papers, including \cite{power graph, power graph 2, cameron3, on the structure, sitov, drugi kelarev i kvin, kelarev3, kelarev4, cameron4,samostalni,panda i krishna, anita i radzkumar, feng ma i vang, masa asrafi i arani, cameron-manna-mehatari, chatto-patra-sahoo, hamzeh-ashrafi, ma-feng-wang, ma-fu-lu, pourghobadi-jafari}. In these papers, combinatorial and algebraic properties of the power graph have received considerable attention, as well as the relation between a group and its power graph. For more details, the survey \cite{pregled za power grafove} is recommended.

Cameron \cite{power graph 2} proved that two finite groups that have isomorphic power graphs also have isomorphic directed power graphs. Cameron, Guerra and Jurina \cite{cameron3} proved that, for torsion-free groups $\mathbf G$ and $\mathbf H$ of nilpotency class $2$, $\mathcal G(\mathbf G)\cong\mathcal G(\mathbf H)$ implies $\vec{\mathcal G}(\mathbf G)\cong\vec{\mathcal G}(\mathbf H)$. The authors in \cite{cameron3} also asked whether this is also true when at least one of the groups is torsion-free of nilpotency class $2$. In \cite{samostalni} was given the affirmative answer to their question. This paper deals further with this problem. Here, the result from \cite{samostalni} is extended by proving that any two torsion-free groups that have isomorphic power graphs have isomorphic directed power graphs, too. Moreover, $\mathcal G(\mathbf G)\cong\mathcal G(\mathbf H)$ implies $\vec{\mathcal G}(\mathbf G)\cong\vec{\mathcal G}(\mathbf H)$ whenever $\mathbf G$ has no quasicyclic subgroup $\mathbf C_{p^{\infty}}$ which has trivial intersection with every cyclic subgroup of $\mathbf G$ not being contained in $\mathbf C_{p^{\infty}}$.

As mentioned before, various combinatorial properties of the power graph have been studied by many authors, although such investigations are not a part of this paper. Aalipour et al. \cite{on the structure} showed that the power graph of every group of a bounded exponent is perfect. They showed that the clique number of the power graph of any group is at most countable, and they posed the question of whether the power graph of every group has at most countable chromatic number. Shitov \cite{sitov} gave the affirmative answer to that question by proving that all power-associative groupoids, i.e. groupoids whose all one-generated subgroupoids are semigroups, have power graphs of at most countable chromatic numbers. Even though his result was combinatorial, it is one of his observations that affected some of the important proofs from this paper and \cite{samostalni}.

\section{Basic Notions and Notations}

{\bf Graph} $\Gamma$ is a structure $(V(\Gamma),E(\Gamma))$, or shortly $(V,E)$, where $V$ is a set, and where $E$ is a set of two-element subsets of $V$. Elements of $V$ are called vertices, and elements of $E$ are called edges of the graph $\Gamma$. Two edges $x$ and $y$ of $\Gamma$ are adjacent if $\{x,y\}\in E$. If $x$ and $y$ are adjacent, we write $x\sim_\Gamma y$, or shortly $x\sim y$. A graph $\Delta$ is a {\bf subgraph} of $\Gamma$ if the vertex set and the edge set of $\Delta$ are subsets of $V(\Gamma)$ and $E(\Gamma)$, respectively. We say that $\Delta$ is an {\bf induced subgraph} of $\Gamma$ if $V(\Delta)\subseteq V(\Gamma)$ and if, for any $x,y\in V(\Delta)$, $x\sim_\Gamma y$ if and only if $x\sim_\Delta y$. In this case we also say that $\Delta$ is the subgraph of $\Gamma$ induced by the set $V(\Delta)$, and the subgraph of $\Gamma$ induced by a set of vertices $X\subseteq V(\Gamma)$ is denoted by $\Gamma[X]$. The {\bf complement} of a graph $\Gamma$ is the graph $\overline \Gamma$ with the same vertex set as $\Gamma$ such that $x\sim_{\overline \Gamma}y$ if and only if $x\not\sim_\Gamma y$.

{\bf Directed graph}, or {\bf digraph}, $\vec\Gamma$ is a structure $(V(\vec\Gamma),E(\vec\Gamma))$, or shortly $(V,E)$, where $V$ is a set, whose elements are called vertices of $\vec\Gamma$, and where $E$ is a set of ordered pairs of different vertices of $\vec\Gamma$. If $(x,y)\in E$, then we say that $x$ is a direct predecessor of $y$, and that $y$ is a direct successor of $x$; in this case, we write $x\rightarrow_{\vec\Gamma} y$, or shortly $x\rightarrow y$.

{\bf Closed neighborhood} of a vertex $x$ of a graph $\Gamma$ is the set $\overline N_\Gamma(x)=\{y\mid y\sim_\Gamma x\text{ or }y=x\}$, and we may shortly denote it by $\overline N(x)$. If two vertices $x$ and $y$ of $\Gamma$ have the same closed neighborhood, then we write  $x\equiv_\Gamma y$, or simply $x\equiv y$. {\bf Path} in a graph $\Gamma$ is a sequence of different vertices $x_0,x_1,x_2,\dots,x_n$ such that $x_{i-1}\sim_\Gamma x_i$ for all $i\in\{1,2,\dots, n\}$, and the {\bf length} of this path is $n$. If, for every two vertices $x$ and $y$ of $\Gamma$, there is a path in $\Gamma$ connecting $x$ and $y$, i.e. in which $x=x_0$ and $y=x_n$, then we say that the graph $\Gamma$ is {\bf connected}. {\bf Connected component} of a graph $\Gamma$ is any maximal connected induced subgraph of $\Gamma$. {\bf Distance} between vertices $x$ and $y$ in a connected graph $\Gamma$, denoted by $d(x,y)$, is the minimal length of a path that connects $x$ and $y$. The maximal distance between two vertices of a connected graph is called the {\bf diameter} of the graph $\Gamma$, and it is denoted by $\diam(\Gamma)$. {\bf Clique} of graph $\Gamma$ is a set of its vertices which induces a complete subgraph of $\Gamma$, where a graph is complete if it has no pair of different non-adjacent vertices.

All through this paper, algebraic structures such as groups are denoted by bold capitals, and their universes are denote by respective regular capital letters. For elements $x$ and $y$ of a group $\mathbf G$ we write $x\approx_{\mathbf G} y$, or simply $x\approx y$, if $\langle x\rangle=\langle y\rangle$, where $\langle x\rangle$ denotes the subgroup generated by $x$. We denote the order of an element $x$ of a group by $o(x)$. We say that a subgroup $\mathbf H$ of $\mathbf G$ is {\bf intersection-free} if $\mathbf H\cap\mathbf K$ is trivial for all cyclic subgroups $\mathbf K$ of $\mathbf G$ such that $\mathbf K\not\leq \mathbf H$. In this paper, we deal with the power graph and the directed power graph of a group, and now we introduce the definitions of these graphs.

\begin{definition}\label{definicije pridruzenih grafova}
The {\bf directed power graph} of a group $\mathbf G$ is the digraph $\vec{\mathcal G}(\mathbf G)$ whose vertex set is $G$, and in which there is a directed edge from $x$ to $y$, $x\neq y$, if there exists $n\in\mathbb Z$ such that $y=x^n$. If there is a directed edge from $x$ to $y$ in $\vec{\mathcal G}(\mathbf G)$, we write $x\rightarrow_{\mathbf G} y$, or shortly $x\rightarrow y$.

The {\bf power graph} of a group $\mathbf G$ is the graph $\mathcal G(\mathbf G)$ whose vertex set is $G$, and whose vertices $x$ and $y$, $x\neq y$, are adjacent if there exists $n\in\mathbb Z$ such that $y=x^n$ or $x=y^n$. If $x$ and $y$ are adjacent in $\mathcal G(\mathbf G)$, we write $x\simp_{\mathbf G} y$, or shortly $x\simp y$.
\end{definition}

Throughout this paper, instead of dealing with the power graph of a group as defined above, it will be more convenient to state our arguments for the power graph as defined in the ensuing definition. To avoid any ambiguity, the power graph as defined in the following definition we call the $Z^\pm$-power graph.

\begin{definition}
The {\bf directed $Z^\pm$-power graph} of a group $\mathbf G$ is the digraph $\vec{\mathcal G}^{\pm}(\mathbf G)$ whose vertex set is $G$, and in which there is a directed edge from $x$ to $y$, $x\neq y$, if there exists $n\in\mathbb Z\setminus\{0\}$ such that $y=x^n$. If there is a directed edge from $x$ to $y$ in $\vec{\mathcal G}^\pm(\mathbf G)$, we write $x\rightarrowpm_{\mathbf G} y$, or shortly $x\rightarrowpm y$.

The {\bf $Z^\pm$-power graph} of a group $\mathbf G$ is the graph $\mathcal G^{\pm}(\mathbf G)$ whose vertex set is $G$, and in which $x$ and $y$, $x\neq y$, are adjacent if there exists $n\in\mathbb Z\setminus\{0\}$ such that $y=x^n$ or $x=y^n$. If $x$ and $y$ are adjacent in $\mathcal G^\pm(\mathbf G)$, we write $x\simppm_{\mathbf G} y$, or shortly $x\simppm y$.
\end{definition}

If elements $x$ and $y$ of a group $\mathbf G$ have the same closed neighborhood in $\mathcal G^\pm(\mathbf G)$, we write $x\equiv_{\mathbf G} y$, or simply $x\equiv y$. By the following theorem, which was proved in \cite{samostalni}, the results obtained in this paper about the $Z^\pm$-power graph of a group apply for the power graph, too.

\begin{theorem}[{\cite[Theorem 5]{samostalni}}]\label{medjusobno odredjivanje opsteg i nenula-stepenog grafa}
Let $\mathbf G$ and $\mathbf H$ be groups. Then $\mathcal G(\mathbf G)\cong\mathcal G(\mathbf H)$ if and only if $\mathcal G^{\pm}(\mathbf G)\cong\mathcal G^{\pm}(\mathbf H)$. 
\end{theorem}

\section{The Power Graph and the Directed Power Graph of a Group}\label{prva sekcija}

In this section, we prove that, if no quasicyclic subgroup of a group is intersection-free, then its power graph determines the directed power graph of the group. Consequently, the power graph of any torsion-free group determines the directed power graph of a group.  By Theorem \ref{medjusobno odredjivanje opsteg i nenula-stepenog grafa}, the power graph and the $Z^\pm$-power graph determine each other up to isomorphism. Therefore, it is justified to provide all proofs in this section for the $Z^\pm$-power graph and the directed $Z^\pm$-power graph instead of the power graph and the directed power graph.

For a group $\mathbf G$, $G_{<\infty}$ and $G_\infty$ denote the set of all elements of finite order of $\mathbf G$ and the set of all elements of infinite order, respectively. No element of infinite order of a group $\mathbf G$ is adjacent in the graph $\mathcal G^\pm(\mathbf G)$ to an element of finite order, and the identity element is adjacent to all non-identity element of finite order. Therefore, $G_{<\infty}$ induces a connected component of $\mathcal G^\pm(\mathbf G)$ (and $\vec{\mathcal G}^\pm(\mathbf G)$). The subgraph of $\mathcal G^\pm(\mathbf G)$ or $\vec{\mathcal G}^\pm(\mathbf G)$ induced by $G_{<\infty}$ we  call the {\bf finite-order component} of $\mathcal G^\pm(\mathbf G)$ or $\vec{\mathcal G}^\pm(\mathbf G)$, respectively. Similarly, the subgraph of $\mathcal G^\pm(\mathbf G)$ or $\vec{\mathcal G}^\pm(\mathbf G)$ induced by $G_{\infty}$ is called the {\bf infinite-order component} of $\mathcal G^\pm(\mathbf G)$ or $\vec{\mathcal G}^\pm(\mathbf G)$, respectively. Notice that, while the finite-order component of the $Z^\pm$-power graph is also its connected component, this may not be the case with the infinite-order component.

By the following lemma, for any group $\mathbf G$, an isomorphism from $\mathcal G^\pm(\mathbf G)$ onto $\mathcal G^\pm(\mathbf H)$ maps the finite-order component of $\mathcal G^\pm(\mathbf G)$ onto the finite-order component of $\mathcal G^\pm(\mathbf H)$.

\begin{lemma}\label{konacni se slikaju na konacne}
Let $\mathbf G$ and $\mathbf H$ be groups, and let $\varphi:G\rightarrow H$ be an isomorphism from $\mathcal G^\pm(\mathbf G)$ onto $\mathcal G^\pm(\mathbf H)$. Then $\varphi(G_{<\infty})=H_{<\infty}$.
\end{lemma}

\begin{proof}
Let $D$ induce a connected component of $\mathcal G^\pm(\mathbf H)$ that contains only elements of infinite order. Then, if $y\in D$, the set $\bigcup_{i\in\mathbb Z\setminus\{0\}}\{y^{2^i},y^{-(2^i)}\}$ is a clique which is a union of $\equiv_{\mathbf H}$-classes of cardinality $2$. On the other hand, the finite-order component of $\mathcal G^\pm(\mathbf G)$ does not contain such a clique. Therefore, any isomorphism from  $\mathcal G^\pm(\mathbf G)$ onto $\mathcal G^\pm(\mathbf H)$ maps $G_{<\infty}$ onto $H_{<\infty}$, which proves the lemma.
\end{proof}\\

The previous lemma justifies us to split the proof of the main result of this paper into two subsections, one in which we deal with isomorphisms between the infinite-order components, and the other in which we deal with isomorphisms between the finite-order components of the power graphs of two groups.

\subsection{Isomorphism between Infinite-Order Components}

In this subsection, it is proved that, if two groups $\mathbf G$ and $\mathbf H$ have isomorphic $Z^\pm$-power graphs, then infinite-order components of $\vec{\mathcal G}^\pm(\mathbf G)$ and $\vec{\mathcal G}^\pm(\mathbf H)$ are isomorphic too. Naturally, this also proves that any two torsion-free groups that have isomorphic $Z^\pm$-power graphs also have isomorphic directed $Z^\pm$-power graphs.

The next lemma tells us an important relationship between any two elements of infinite order belonging to the same connected component of the $Z^\pm$-power graph of a group.

\begin{lemma}\label{sve ciklicne u komponenti se seku}
Let $\mathbf G$ be a  group. Then, for any $x$ and $y$ belonging to the same connected component of the infinite-order component of $\mathcal G^\pm(\mathbf G)$, subgroups $\langle x\rangle$ and $\langle y\rangle$ have a non-trivial intersection.
\end{lemma}

\begin{proof}
Let $\Gamma=\mathcal G^\pm(\mathbf G)$, and let $C\subseteq G_{\infty}$ induce a connected component of $\Gamma$. Let us prove first that $\diam(\Gamma[C])=2$. Let $x,y\in C$ be such that $d_\Gamma(x,y)>2$. Therefore, in the path of minimal length from $x$ to $y$, there are consecutive elements $a$, $b$ and $c$ such that $a\leftarrowpm b$ and $b\rightarrowpm c$. Then, $\langle a\rangle\cap\langle c\rangle=\langle b^n\rangle$ for some $n\in\mathbb N$. Thus, we can make a shorter path from $x$ to $y$ by replacing vertices $a$ and $b$ or vertices $b$ and $c$ by vertex $b^n$. This way we get a shorter path from $x$ to $y$, which is a contradiction.

Finally, because $\diam(\Gamma[C])=2$, it follows that $\langle x\rangle\cap\langle y\rangle$ is non-trivial for any $x$ and $y$ that belong to the same connected component of the infinite-order component of $\mathcal G^\pm(\mathbf G)$.
\end{proof}\\

For an element $x$ of a group $\mathbf G$, let us define sets $I_{\mathbf G}(x)$, $O_{\mathbf G}(x)$ and $M_{\mathbf G}(x)$ as follows:
\begin{align*}
&I_{\mathbf G}(x)=\{y\in V\setminus\{x^{-1}\}\mid y\rightarrowpm_{\mathbf G} x\},\\
&O_{\mathbf G}(x)=\{y\in V\setminus\{x^{-1}\}\mid x\rightarrowpm_{\mathbf G} y\}\text{ and}\\
&M_{\mathbf G}(x)=I_{\mathbf G}(x)\cup O_{\mathbf G}(x).
\end{align*}
Sometimes we may shortly denote them by $I(x)$, $O(x)$ and $M(x)$, respectively. Furthermore, for a group $\mathbf G$ and its $Z^\pm$-power graph $\Gamma^\pm=\mathcal G^\pm(\mathbf G)$, $\overline{\mathcal O}_{\mathbf G}(x)$ and $\overline{\mathcal M}_{\mathbf G}(x)$ denote $\overline{\Gamma^\pm}[O(x)]$ and $\overline{\Gamma^\pm}[M(x)]$, respectively. Note that, for an element $x$ of infinite order of $\mathbf G$, one can recognize the element $x^{-1}$ as the only vertex which has the same closed neighborhood in $\mathcal G^\pm(\mathbf G)$ as the vertex $x$.

The following lemma has been proved by Cameron, Guerra and Jurina \cite{cameron3}. Although in the original paper it was proved for an element of a torsion-free group, it is proved analogously for an element of infinite order of any group.

\begin{lemma}[{\cite[Lemma 3.3]{cameron3}}]\label{O je komponenta povezanosti od N}
Let $\mathbf G$ be a group, and let $x$ be an element of infinite order of group $\mathbf G$. Then $\overline{\mathcal O}_{\mathbf G}(x)$ is a connected component of $\overline{\mathcal M}_{\mathbf G}(x)$.
\end{lemma}


To prove the main result of this paper, we shall start by Proposition \ref{komponenta gde vidimo pravac}, which deals with certain connected components of the infinite-order component of the $Z^\pm$-power graph in which it is possible to reconstruct the directions of arcs of the directed $Z^\pm$-power graph.

We say that a graph is {\bf almost connected} if it is a disjoint union of a connected graph and two copies of the trivial graph $K_1$. For a group $\mathbf G$, suppose that $C$ induces a connected component of $\mathcal G^\pm(\mathbf G)$ which contains only elements of infinite order, and let $x$ and $y$ be elements of $C$ non-adjacent in $\mathcal G^\pm(\mathbf G)$. Then $\overline{\mathcal O}(x)\cap\overline{\mathcal O}(y)$ is an almost connected graph. Namely, by Lemma \ref{sve ciklicne u komponenti se seku}, $\langle x\rangle\cap\langle y\rangle$ is an infinite cyclic subgroup of $\mathbf G$. Therefore, $\langle x\rangle\cap\langle y\rangle=\langle z\rangle$ for some $z\in C$, and $O(x)\cap O(y)=\langle z\rangle\setminus\{e\}=O(z)\cup \{z,z^{-1}\}$. Now we see that $z$ and $z^{-1}$ are isolated vertices of $\overline{\mathcal O}(x)\cap\overline{\mathcal O}(y)$. Also, $\big(O(x)\cap O(y)\big)\setminus\{z,z^{-1}\}=O(z)$ induces a connected subgraph of $\overline{\mathcal O}(x)\cap\overline{\mathcal O}(y)$ because, for any $n,m\in\mathbb Z\setminus\{-1,0,1\}$, there is $k>1$ relatively prime to both $n$ and $m$, and, therefore, vertices $z^n$ and $z^m$ are connected in $\overline{\mathcal O}(x)\cap\overline{\mathcal O}(y)$ with the path $z^n\simppm z^k\simppm z^m$. This observation will be useful in proofs of Proposition \ref{komponenta gde vidimo pravac} and Proposition \ref{komponenta gde mozda ne vidimo pravac}.

\begin{proposition}\label{komponenta gde vidimo pravac}
Let $\mathbf G$ and $\mathbf H$ be groups. Let $\varphi: G\rightarrow H$ be an isomorphism from $\mathcal G^\pm(\mathbf G)$ onto $\mathcal G^\pm(\mathbf H)$, and let $C\subseteq G$ induce a connected component of $\mathcal G^\pm(\mathbf G)$ which contains only elements of infinite order. If there are $x,y\in C$, $x\not\simppm_{\mathbf G} y$, such that $\overline{\mathcal M}_{\mathbf G}(x)\cap \overline{\mathcal M}_{\mathbf G}(y)$ is an almost connected graph, then $\varphi\rvert_C$ is an isomorphism  from $\big(\vec{\mathcal G}^\pm(\mathbf G)\big)[C]$ onto $\big(\vec{\mathcal G}^\pm(\mathbf H)\big)[\varphi(C)]$.
\end{proposition}

\begin{proof}
Let us denote $\varphi(C)$ by $D$. By Lemma \ref{konacni se slikaju na konacne}, $D$ contains only elements of infinite order. By Lemma \ref{sve ciklicne u komponenti se seku}, then $\emptyset\neq O_{\mathbf G}(x)\cap O_{\mathbf G}(y)\subseteq M_{\mathbf G}(x)\cap M_{\mathbf G}(y)$. Further, $I_{\mathbf G}(x)\cap M_{\mathbf G}(y)=I_{\mathbf G}(y)\cap M_{\mathbf G}(x)=\emptyset$, because otherwise, by Lemma \ref{O je komponenta povezanosti od N}, $\overline{\mathcal M}_{\mathbf G}(x)\cap \overline{\mathcal M}_{\mathbf G}(y)$ would not be almost connected. Therefore, $M_{\mathbf G}(x)\cap M_{\mathbf G}(y)= O_{\mathbf G}(x)\cap O_{\mathbf G}(y)$. Similarly, $M_{\mathbf H}\big(\varphi(x)\big)\cap M_{\mathbf H}\big(\varphi(y)\big)=O_{\mathbf H}\big(\varphi(x)\big)\cap O_{\mathbf H}\big(\varphi(y)\big)$.

Suppose that $u\rightarrowpm_{\mathbf G} v$ for some $u,v\in C$. If $v=u^{-1}$, then $u\equiv_{\mathbf G} v$. This would imply that $\varphi(u)\equiv_{\mathbf H}\varphi(v)$, and that $\varphi(v)=(\varphi(u))^{-1}$. So suppose that $v\neq u^{-1}$. Let us prove that $\varphi(u)\rightarrowpm_{\mathbf H} \varphi(v)$. Because $v\in O_{\mathbf G}(u)$ and by Lemma \ref{sve ciklicne u komponenti se seku}, the connected component of $\overline{\mathcal M}_{\mathbf G}(u)$ which contains $v$ has infinite intersection with $M_{\mathbf G}(x)\cap M_{\mathbf G}(y)$. It follows that the connected component of $\overline{\mathcal M}_{\mathbf H}\big(\varphi(u)\big)$ which contains $\varphi(v)$ has infinite intersection with $M_{\mathbf H}\big(\varphi(x)\big)\cap M_{\mathbf H}\big(\varphi(y)\big)$, which implies that $\varphi(u)\rightarrowpm_{\mathbf H}\varphi(v)$. It is analogously proved that $\varphi(u)\rightarrowpm_{\mathbf H}\varphi(v)$ implies $u\rightarrowpm_{\mathbf G} v$. Therefore, the mapping $\varphi\lvert_C$ is an isomorphism  from $\big(\vec{\mathcal G}^\pm(\mathbf G)\big)[C]$ onto $\big(\vec{\mathcal G}^\pm(\mathbf H)\big)[\varphi(C)]$.
\end{proof}\\

In the remained of this subsection, we deal with the rest of the connected components of the infinite-order component of the $Z^\pm$-power graph of a group. The following theorem, which was proved in \cite{samostalni}, will serve as a useful tool here.


%
%
%


%
%

\begin{theorem}[{\cite[Theorem 21]{samostalni}}]\label{odredjivanje usmerenog kod torziono slobodne klase nilpotentnosti 2}
Let $\mathbf G$ be a torsion-free group of nilpotency class $2$, and let $\mathbf G$ be a group such that $\mathcal G^\pm(\mathbf G)\cong \mathcal G^\pm(\mathbf H)$. Then $\vec{\mathcal G}^\pm(\mathbf G)\cong\vec{\mathcal G}^\pm(\mathbf H)$.
\end{theorem}

\begin{proposition}\label{komponenta gde mozda ne vidimo pravac}
Let $\mathbf G$ and $\mathbf H$ be groups. Let $\varphi: G\rightarrow H$ be an isomorphism from $\mathcal G^\pm(\mathbf G)$ onto $\mathcal G^\pm(\mathbf H)$, and let $C\subseteq G$ induce a connected component of $\mathcal G^\pm(\mathbf G)$ which contains only elements of infinite order. If $\overline{\mathcal M}_{\mathbf G}(x)\cap \overline{\mathcal M}_{\mathbf G}(y)$ is an almost connected graph for no pair of elements $x,y\in G$ such that $x\not\simppm_{\mathbf G} y$, then $\big(\vec{\mathcal G}^\pm(\mathbf G)\big)[C]\cong\big(\vec{\mathcal G}^\pm(\mathbf H)\big)[\varphi(C)]$.
\end{proposition}

\begin{proof}
Let $D$ denote $\varphi(C)$, which, by Lemma \ref{konacni se slikaju na konacne}, contains only elements of infinite order. Let $x,y\in C$, and suppose that $x\not\simppm_{\mathbf G}y$.  Because $\overline{\mathcal M}_{\mathbf G}(x)\cap \overline{\mathcal M}_{\mathbf G}(y)$ is not almost connected, and because $O(x)\cap I(y)\neq\emptyset$ or $I(x)\cap O(y)\neq\emptyset$ would imply $x\simppm_{\mathbf G}y$, it follows that $I_{\mathbf G}(x)\cap I_{\mathbf G}(y)\neq\emptyset$. Therefore, there is $z\in C$ such that $z\rightarrowpm_{\mathbf G} x$ and $z\rightarrowpm_{\mathbf G} y$, i.e. $x,y\in\langle z\rangle$. It follows that $\langle x,y\rangle\subseteq\langle z\rangle\subseteq C\cup\{e\}$. Also, if $x\simppm_{\mathbf G}y$, then $\langle x,y\rangle=\langle y\rangle\subseteq C\cup\{e\}$ or $\langle x,y\rangle=\langle x\rangle\subseteq C\cup\{e\}$. Thus, $C\cup\{e\}$ is the universe of a locally cyclic torsion-free subgroup $\hat{\mathbf C}$ of the group $\mathbf G$. Similarly, $D\cup \{e\}$ is the universe of a locally cyclic torsion-free subgroup $\hat{\mathbf D}$ of the group $\mathbf H$. It follows that $\mathcal G^\pm(\hat{\mathbf C})\cong\mathcal G^\pm(\hat{\mathbf D})$. Then, by Theorem \ref{odredjivanje usmerenog kod torziono slobodne klase nilpotentnosti 2} and because $\hat{\mathbf C}$ and $\hat{\mathbf D}$ are abelian, we have $\vec{\mathcal G}^\pm(\hat{\mathbf C})\cong\vec{\mathcal G}^\pm(\hat{\mathbf D})$. Therefore, $\big(\vec{\mathcal G}^\pm(\mathbf G)\big)[C]\cong\big(\vec{\mathcal G}^\pm(\mathbf H)\big)[D]$, which finishes our proof.
\end{proof}\\

Now, with Proposition \ref{komponenta gde vidimo pravac} and Proposition \ref{komponenta gde mozda ne vidimo pravac} on our hands, we can prove the main theorem of this subsection.

\begin{theorem}\label{izomorfni podgrafovi sa elementima beskonacnog reda}
Let $\mathbf G$ and $\mathbf H$ be groups whose $Z^\pm$-power graphs have isomorphic infinite-order components. Then the directed $Z^\pm$-power graphs of $\mathbf G$ and $\mathbf H$ have isomorphic infinite-order components too.
\end{theorem}

\begin{proof}
 Let $C$ be a connected component of the infinite-order component of $\mathcal G^\pm(\mathbf G)$. If  there are some elements $x,y\in C$ non-adjacent in $\mathcal G^\pm(\mathbf G)$ for which $\overline{\mathcal M}_{\mathbf G}(x)\cap \overline{\mathcal M}_{\mathbf G}(y)$ is almost connected, then, by Proposition \ref{komponenta gde vidimo pravac}, it follows that $\big(\vec{\mathcal G}^\pm(\mathbf G)\big)[C]\cong\big(\vec{\mathcal G}^\pm(\mathbf G)\big)[\varphi(C)]$. If $\overline{\mathcal M}_{\mathbf G}(x)\cap \overline{\mathcal M}_{\mathbf G}(y)$ is almost connected for no pair of different elements $x$ and $y$ from $C$ such that $x\not\simppm_{\mathbf G}y$, then, by Proposition \ref{komponenta gde mozda ne vidimo pravac}, we get $\big(\vec{\mathcal G}^\pm(\mathbf G)\big)[C]\cong\big(\vec{\mathcal G}^\pm(\mathbf G)\big)[\varphi(C)]$. From this follows that the directed $Z^\pm$-power graphs of groups $\mathbf G$ and $\mathbf H$ have isomorphic infinite-order components.
\end{proof}

\begin{corollary}
Let $\mathbf G$ be a torsion-free group, and $\mathbf H$ be a group such that $\mathcal G^\pm(\mathbf G)\cong\mathcal G^\pm(\mathbf H)$. Then $\vec{\mathcal G}^\pm(\mathbf G)\cong\vec{\mathcal G}^\pm(\mathbf H)$.
\end{corollary}

\begin{proof}
By Lemma \ref{konacni se slikaju na konacne},  $\mathcal G^\pm(\mathbf G)$ and $\mathcal G^\pm(\mathbf H)$ have isomorphic infinite-order components and isomorphic finite-order components. Because $\mathbf G$ is torsion-free, the finite-order component of $\mathcal G^\pm(\mathbf G)$ has only one vertex, and so the same holds for $\mathcal G^\pm(\mathbf H)$. Therefore, to prove that $\vec{\mathcal G}^\pm(\mathbf G)\cong\vec{\mathcal G}^\pm(\mathbf H)$, it is sufficient to show that $\vec{\mathcal G}^\pm(\mathbf G)$ and $\vec{\mathcal G}^\pm(\mathbf H)$ have isomorphic infinite-order components. But, by Theorem \ref{izomorfni podgrafovi sa elementima beskonacnog reda}, $\vec{\mathcal G}^\pm(\mathbf G)$ and $\vec{\mathcal G}^\pm(\mathbf H)$ do have isomorphic infinite-order components. Thus, the corollary has been proved.
\end{proof}\\

Now, the subsequent statement follows directly by Theorem \ref{medjusobno odredjivanje opsteg i nenula-stepenog grafa}.

\begin{corollary}
Let $\mathbf G$ be a torsion-free group, and $\mathbf H$ be a group such that $\mathcal G(\mathbf G)\cong\mathcal G(\mathbf H)$. Then $\vec{\mathcal G}(\mathbf G)\cong\vec{\mathcal G}(\mathbf H)$.
\end{corollary}

\subsection{Isomorphism between Finite-Order Components}

In this subsection, we give the proof that, if two groups have isomorphic $Z^\pm$-power graphs, and if at least one of them does not contain any intersection-free quasicyclic subgroup, then the finite-order components of their directed $Z^\pm$-power graphs are also isomorphic. Proofs from this subsection rely on the ideas presented in \cite{power graph 2} by Peter Cameron, where he showed that the power graph of a finite group $\mathbf G$ determines the directed power graph. There he noticed that it is possible to determine the directions of arcs between vertices from different $\equiv_{\mathbf G}$-classes. He also observed that, although it may be impossible to determine directions of all arcs within a single $\equiv_{\mathbf G}$-class, it is possible to determine the induced subgraph of $\vec{\mathcal G}^\pm(\mathbf G)$ by that $\equiv_{\mathbf G}$-class up to isomorphism. The difference here is that the set of all elements of finite order of a group may not be finite, and it may not even be a universe of a subgroup of the group.

The following proposition is a generalization of \cite[Proposition 4]{power graph 2}.  Notice that the finite-order component of the $Z^\pm$-power graph of a group has at least one vertex adjacent to all other vertices. Therefore, for a group $\mathbf G$, the set of all vertices of the finite-order component $\Phi$ of $\mathcal G^\pm(\mathbf G)$ adjacent to all other vertices of $\Phi$ we shall call the {\bf center} of $\Phi$, and we will denote it by $\Cen(\Phi)$. More formally, $\Cen(\Phi)=\{x\in G_{<\infty}\mid x\simppm_{\mathbf G} y\text{ for all } y\in G_{<\infty}\}$. As a result of Proposition \ref{kameronova propozicija}, for a group $\mathbf G$, we will be able to prove more easily that the finite-order component of $\mathcal G^\pm(\mathbf G)$ determines the finite-order component of $\vec{\mathcal G}^\pm(\mathbf G)$ when the center of the finite-order component of $\mathcal G^\pm(\mathbf G)$ contains more than one element, while, in this subsection, we will mostly deal with the case when the center of $\mathcal G^\pm(\mathbf G)$ is trivial.

\begin{proposition}\label{kameronova propozicija}
Let $\mathbf G$ be a group such that $\lvert\Cen(\Phi)\rvert>1$, where $\Phi$ is the finite-order component of $\mathcal G^\pm(\mathbf G)$, and let $S=\Cen(\Phi)$. Then one of the following holds:
\begin{enumerate}
\item $\mathbf G_{<\infty}$ is a Pr\" ufer group. In this case, $S=G_{<\infty}$ and $S$ is infinite. 
\item $\mathbf G_{<\infty}$ is a cyclic group of prime power order. In this case, $S=G_{<\infty}$ and $S$ is finite.
\item $\mathbf G_{<\infty}$ is a cyclic group whose order is the product of two different prime numbers. In this case, $\lvert S\rvert\geq \frac 12\lvert G_{<\infty}\rvert$ and the set $G_{<\infty}\setminus S$ induces a disconnected subgraph of $\Phi$.
\item $\mathbf G_{<\infty}$ is a cyclic group whose order is divisible by at least two different prime numbers, but whose order is not the product of two different prime numbers. In this case, the set $G_{<\infty}\setminus S$ induces a connected subgraph of $\Phi$.
\item There is a prime number $p$ such that the order of every element from $G_{<\infty}$ is a power of $p$, but $\langle G_{<\infty}\rangle$ is not a cyclic, nor a Pr\" ufer group. In this case, $\lvert S\rvert<\frac 12\lvert G_{<\infty}\rvert$ and the set $G_{<\infty}\setminus S$ induces a disconnected subgraph of $\Phi$.
\end{enumerate}
\end{proposition}

\begin{proof}
Let $\mathcal P$ be the set of all prime numbers $p$ such that $G_{<\infty}$ contains an element of order $p$. In this proof, by the exponent of a subset $X$ of $G$ we mean the least $k\in\mathbb N$ such that $x^k=e$ for all $x\in X$.

Suppose first that the set $\mathcal P$ contains only one prime number. If $\mathbf G_{<\infty}$ is a cyclic group, then $S=G_{<\infty}$, and $S$ is finite. If $\mathbf G_{<\infty}$ is a Pr\" ufer group, then $S=G_{<\infty}$ and $S$ is infinite. Suppose further that $G_{<\infty}$ is not the universe of a cyclic subgroup of $\mathbf G$, nor it is the universe of a subgroup of $\mathbf G$ isomorphic to a Pr\" ufer group. Then $\langle G_{<\infty}\rangle$ is not a cyclic nor a Pr\" ufer group, because $G_{<\infty}$ already contains all elements of finite order of $\mathbf G$. Let us show that there is an element of $S$ of maximal order. Let $x\in S$ and $y\in G_{<\infty}\setminus S$. Then $y\rightarrowpm x$, because $\langle x\rangle\subseteq S$. Now, if $S$ had no element of maximal order, then there would be no element $y$ such that $y\rightarrowpm x$ for all $x\in S$, which, by the above discussion, would imply that $S=G_{<\infty}$. This is a contradiction with the fact that $\langle G_{<\infty}\rangle$ is not isomorphic to a Pr\" ufer group. Thus, there is an element $x\in S$ of maximal order, and $\langle x\rangle=S$. Let $o(x)=p^k$ for some $k\in\mathbb N$. Then there are $y,z\in G_{<\infty}\setminus S$ of order $p^{k+1}$ such that $y\not\approx z$. Then $\overline N_{\Phi[G_{<\infty}\setminus S]}(y)$ and $\overline N_{\Phi[G_{<\infty}\setminus S]}(z)$ are different connected components of $\Phi[G_{<\infty}\setminus S]$, and their cardinalities are at least $(p-1)p^k$. Therefore, $\lvert S\rvert<\frac 12\lvert G_{<\infty}\rvert$.

Suppose now that $\lvert\mathcal P\rvert>1$. Let $x\in S\setminus\{e\}$. For every $i\leq n$, $o(x)$ is divisible by $p$, because otherwise $x$ would not be adjacent to any element of order $p$. Therefore, $\mathcal P$ is finite, and the exponent of $G_{<\infty}$ is $p_1^{k_1}p_2^{k_2}\cdots p_m^{k_m}$ for some $m>1$ and for some prime numbers $p_1,p_2,\dots,p_m$. Moreover, for any $i\leq m$, there is an element $y_i\in G_{<\infty}$ of order $p_i^{k_i}$. Because $x\rightarrowpm y_i$, then $p_i^{k_i}\mid o(x)$. Thus, $o(x)$ is equal to the exponent of $G_{<\infty}$, and $\mathbf G_{<\infty}=\langle x\rangle$ because $x\in S\setminus\{e\}$. Now, if $o(x)$ is not a product of two different prime numbers, then the graph $\Phi[G_{<\infty}\setminus S]$ is connected. However, if $o(x)=pq$, for different prime numbers $p$ and $q$, then $\Phi[G_{<\infty}\setminus S]$ is disconnected and
\begin{align*}
\lvert S\rvert&=(p-1)(q-1)+1=\frac{pq+pq-2p-2q+4}2\\
&=\frac{pq}2+\frac{(p-2)(q-2)}2\geq\frac {pq}2=\frac{\lvert G\rvert}2.
\end{align*}
This proves the proposition.
\end{proof}\\

Let us show now that, if we knew all $\approx$-classes and their relations in the $Z^\pm$-power graph, then, for any two adjacent $\approx$-classes which contain elements of finite order, it would be possible to  determine which one of them contains elements of greater order. The following fact was used in \cite{power graph 2}, although it was not given as a separate lemma there.

\begin{lemma}\label{odredjivanje pravaca}
Let $\mathbf G$ be a group, and let $x,y\in G_{<\infty}$, $x\neq y$. Then $x\rightarrowpm y$ if and only if at least one of the following holds:
\begin{enumerate}
\item $x\simppm y$ and $\big\lvert[y]_\approx\big\rvert<\big\lvert[x]_\approx\big\rvert$;
\item $x\simppm y$, $\big\lvert[y]_\approx\big\rvert=\big\lvert[x]_\approx\big\rvert$, and $x\simppm z$ for some $z\in G_{<\infty}$ such that $[z]_\approx=\{z\}$, and $\overline N(z)\neq G$;
\item $x\approx y$.
\end{enumerate}
\end{lemma}

\begin{proof}
It is known that, for any $n,m\in\mathbb N$, $n\mid m$ implies that $\varphi(n)\mid\varphi(m)$, where $\varphi$ denotes Euler totient function. Moreover, $n\mid m$ implies $\varphi(n)<\varphi(m)$ unless $m=2n$ for an odd number $n$, or unless $n= m$. But when $m=2n$ for an odd number $n$, then, if an element $x$ of order $n$ is adjacent to an element $y$ of order $m$, the element $y$ is adjacent to an element $z$ of order $2$, while $x$ is adjacent to no such element. Note that, beside the identity element, elements of order $2$ are the only ones contained in one-element $\approx$-classes. Therefore, $x\rightarrowpm y$ if and only if one of the three conditions is fulfilled.
\end{proof}\\

The above lemma will be useful for us, but from the $Z^\pm$-power graph, we do not see $\approx$-classes. The following four lemmas will, with the help of Lemma \ref{odredjivanje pravaca}, enable us to determine directions of arches of the directed $Z^\pm$-power graph between different $\equiv$-classes and to determine directions of arches within $\equiv$-classes up to isomorphism. The following lemma is a generalization of \cite[Proposition 5]{power graph 2}, and it is one of the essential facts for the proof of the main result of this subsection.

\begin{lemma}\label{dva tipa klasa}
Let $\mathbf G$ be a group such that $\lvert\Cen(\Phi)\rvert=1$, where $\Phi$ is the finite-order component of $\mathcal G^\pm(\mathbf G)$.  Then every $\equiv$-class $C$ of $\Phi$ is one of the following forms:
\begin{enumerate}
\item $C$ is an $\approx$-class. Such an $\equiv$-class we call a {\bf simple $\equiv$-class}.
\item $C=\{x\in\langle y\rangle\mid o(x)\geq p^s\}$, where $p$ is a prime number, $y$ is an element of order $p^r$ for some $r\in\mathbb N$, and where $s\in\mathbb N$ satisfies $r>s>0$. In this case, $C$ is a union of $r-s+1$ $\approx$-classes, and we say that such an $\equiv$-class is a {\bf complex $\equiv$-class}.
\item $C=\bigcup_{k\geq s}[x_k]_\approx$ for some $s\geq 1$, where, for some prime number $p$, each $x_k$ is an element of order $p^k$, and where $x_k\in\langle x_{k+1}\rangle$ for all $k\geq s$. Such an $\equiv$-class we call an {\bf infinitely complex} $\equiv$-class.
\end{enumerate}
\end{lemma}

\begin{proof}
It is easily seen that every $\equiv$-class is a union of $\approx$-classes. Also, if all elements of an $\equiv$-class have the same order, then that $\equiv$-class is also an $\approx$-class, i.e. it is a simple $\equiv$-class.

Let $C$ be an $\equiv$-class, and suppose that $C$ contains elements $x$ and $y$ of different orders. Let us prove that the $\equiv$-class $C$ is complex or infinitely complex. Without loss of generality, let $o(x)<o(y)$. Then $o(x)$ is a divisor of $o(y)$ because $y\rightarrowpm x$. Let us show that $o(y)$ is a power of a prime number. If that is not the case, then there are different prime numbers $q_1$ and $q_2$ such that $q_1\mid\frac{o(y)}{o(x)}$ and $q_2\mid o(x)$. Then there is an element of order $\frac{o(x)q_1}{q_2}$ which is adjacent to $y$ and not to $x$, which is a contradiction. Therefore, the order of $y$ is a power of a prime number.

Now we know that there is a prime number $p$ such that $C$ contains only elements whose orders are powers of $p$. Because $\lvert\Cen(\Phi)\rvert=1$, $C$ does not contain $e$, i.e. it does not contain an element of order $p^0$. Further, if $x$ and $y$, such that $o(x)<o(y)$, belong to the same $\equiv$-class $C$, and if $z$ is an element such that $\langle x\rangle\leq\langle z\rangle\leq\langle y\rangle$, then $\overline N(y)\subseteq\overline N(z)\subseteq\overline N(x)$, because $x$, $y$ and $z$ have prime power orders. This implies that $\overline N(z)=\overline N(x)$, i.e. $z\equiv x$. Now, if $C$ contains an element of maximal order, then $C$ is a complex $\equiv$-class. Otherwise, $C$ is an infinitely complex $\equiv$-class. 
\end{proof}\\

In the above lemma, we introduced the notions of simple, complex and infinitely complex $\equiv$-classes. Although Lemma \ref{dva tipa klasa} deals with the case when the center of the finite-order component contains only the identity element of the group, we will use those terms when dealing with the finite-order component of any group.

\begin{lemma}\label{svako svakog tuce}
Let $\mathbf G$ be a group such that $\lvert\Cen(\Phi)\rvert=1$, where $\Phi$ is the finite-order component of $\mathcal G^\pm(\mathbf G)$, and let $x_0,y_0\in G_{<\infty}$. If $x_0\not\equiv y_0$ and $x_0\rightarrowpm y_0$, then $x\rightarrowpm y$ for all $x$ and $y$ such that $x\equiv x_0$ and $y\equiv y_0$.
\end{lemma}

\begin{proof}
Suppose that $[x_0]_\equiv$ is an infinitely complex $\equiv$-class. Then there is no element $z\in G\setminus[x_0]_\equiv$ such that $z\rightarrowpm x_0$. Thus, for any $x\in[x_0]_\equiv$, and for any element $y$ such that $y\simppm x$ and $y\not\equiv x$, it follows that $x\rightarrowpm y$. Now it only remains to prove the lemma in the case when neither of $\equiv$-classes $[x]_\equiv$ and $[y]_\equiv$ is infinitely complex. Suppose further that none of $[x]_\equiv$ and $[y]_\equiv$ is an infinitely complex $\equiv$-class.

It suffices to show that $x\equiv x_0$, $y\equiv y_0$, $x_0\not\equiv y_0$ and $x_0\rightarrowpm y_0$ implies $x\rightarrowpm y_0$ and $x_0\rightarrowpm y$, for any $x_0,x,y,y_0\in G$. Suppose that $x\equiv x_0$, $x_0\not\equiv y_0$ and $x_0\rightarrowpm y_0$. If $[x_0]_\equiv$ is a simple $\equiv$-class, then it is easily seen that the implication holds, so suppose that $[x_0]_\equiv$ is a complex $\equiv$-class. Suppose that $x\not\rightarrowpm y_0$. Then $y_0\rightarrowpm x$, which implies that $\langle x\rangle\leq\langle y_0\rangle\leq\langle x_0\rangle$ and that the orders of $x_0$, $x$ and $y_0$ are powers of a prime number. Therefore, $\overline N(x_0)\subseteq\overline N(y_0)\subseteq \overline N(x)$, and thus $x_0\equiv y_0$, which is a contradiction. This proves that $x\rightarrowpm y_0$. It is proved similarly that $y\equiv y_0$, $x_0\not\equiv y_0$ and $x_0\rightarrowpm y_0$ implies $x_0\rightarrowpm y$. Thus, the lemma has been proved.
\end{proof}\\

The following lemma was proved by Cameron \cite{power graph 2}. Although he did not state it as a separate proposition, it was one of the essential steps in his proof that the power graph of a finite group determines the directed power graph. It is one of the crucial facts for this subsection too.

In the remainder of this subsection, for a set $S\subseteq G_{\infty}$, $\hat S$ shall denote the set
\[\hat S=\overline N(\overline N(S)),\]
where $\overline N(S)=\bigcap_{x\in S}\overline N(x)$.

\begin{lemma}\label{prepoznavanje tipova klasa}
Let $\mathbf G$ be a group such that $\lvert\Cen(\Phi)\rvert=1$, where $\Phi$ is the finite-order component of $\mathcal G^\pm(\mathbf G)$. Let $C$ be a complex $\equiv$-class. Then the following holds:
\begin{enumerate}
\item $|\hat C|=p^r$ and $|\hat C|-|C|=p^{s-1}$ for some $r,s\in\mathbb N$ such that $r>s> 0$;
\item $C$ is adjacent to no mutually non-adjacent $\equiv$-classes $D$ and $E$ such that $|D|,|E|\leq |C|$.
\end{enumerate}
Further, $p^r$ and $p^s$ are the maximum and the minimum order of an element of $C$, respectively.

If $C$ is a simple $\equiv$-class, then at least one of the above statements is not satisfied.
\end{lemma}

\begin{proof}
Suppose first that $C$ is a complex $\equiv$-class, and let $y$ be an element of $C$ of maximal order. Let $o(y)=p^r$ for a prime number $p$. Let us prove that $\hat C=\langle y\rangle$. Because $y$ has prime power order, $\langle y\rangle\subseteq\hat C$. Suppose now that there is an element $z\in \hat C\setminus\langle y\rangle$. Because $\hat C\subseteq\overline N(y)$, we get that $\langle y\rangle <\langle z\rangle$. If $z$ was not of prime power order, then $C$ would be a simple $\equiv$-class. Also, if $z$ is of prime power order, that $\overline N(z)\subseteq\overline N(y)$ because $\langle y\rangle\leq\langle z\rangle$. We also have that $\overline N(y)\subseteq\overline N(z)$ because $z\in \hat C$, and therefore $y\equiv z$, which is a contradiction. Now $\hat C=\langle y\rangle$ implies that the first condition is fulfilled, and that $p^r=\lvert \hat C\rvert$. Also, if $p^{s-1}$ is the maximal order of an element of $\hat C\setminus C$, and if $z$ is an element of $\hat C\setminus C$ of order $p^{s-1}$, then $\langle z\rangle=\hat C\setminus C$. This implies that $|\hat C|-|C|=p^{s-1}$.

Let us prove that, for the complex $\equiv$-class $C$, the second condition is fulfilled too. If $y\rightarrowpm z$, then $\lvert[z]_\equiv\rvert<\lvert C\rvert$, but any such $\equiv$-class is adjacent to all other $\equiv$-classes adjacent to $C$. Now it is sufficient to show that all $\equiv$-classes $[z]_\equiv$ adjacent to $C$, such that $z\rightarrowpm y$, have greater cardinality than $\lvert C\rvert$.  
If $z\rightarrowpm y$, then 
\[\lvert C\rvert<p^r\leq p^r(p-1)=\varphi(p^{r+1})\leq[z]_\equiv,\]
where $\varphi$ is Euler totient function. Therefore, the second condition is fulfilled too.

Let us prove now that a simple $\equiv$-class does not fulfill at least one of the two conditions. Let $C$ be a simple $\equiv$-class such that the order of its elements is divisible by at least two different prime numbers $p$ and $q$. Then there are classes $D$ and $E$ which contain elements of orders $p$ and $q$, respectively. Therefore, $C$ does not fulfill the second condition.

Suppose now that the elements of $C$ are of prime power order. Obviously, if $C=\{e\}$, then the first condition is not satisfied, so suppose further that there are $k\in\mathbb N$ and a prime number $p$ such that all elements of $C$ have order $p^k$. Let $y\in C$, and suppose that $C$ satisfies the first condition. Then
\[\lvert\hat C\rvert=p^r< p^r+(p^r-2p^{s-1})=2(p^r-p^{s-1})= 2\lvert C\rvert.\]
Therefore, $\hat C$ does not contain any element of order greater than $p^k$. Now, in a similar way as in the first paragraph of this proof, it can be shown that $\hat C=\langle y\rangle$, and therefore $C$ does not fulfill the first condition. This proves the lemma.
\end{proof}

\begin{lemma}\label{prepoznavanje redova u beskonacno komplesknoj klasi}
Let $\mathbf G$ be a group such that $\lvert\Cen(\Phi)\rvert=1$, where $\Phi$ is the finite-order component of $\mathcal G^\pm(\mathbf G)$. Let $C$ be an infinitely-complex $\equiv$-class. Then $\lvert\hat C\rvert-\lvert C\rvert =p^{s-1}$ for a prime number $p$ and for some $s\in\mathbb N$, and $p^s$ is the minimal order of an element of $C$.
\end{lemma}

\begin{proof}
Because $C$ is an infinitely complex $\equiv$-class, there is a prime number $p$ such that orders of all elements of $C$ are powers of $p$. Let $p^s$ be the minimal order of an element of $C$. Because $\lvert\Cen(\Phi)\rvert=1$, $C$ does not contain the identity element of the group, and, therefore, $s>0$.

Because $x\rightarrowpm y$ for no elements $y\in C$ and $x\in G_{<\infty}\setminus C$, the set $\overline N(C)$ is the universe of a quasicyclic subgroup $\mathbf C_{p^\infty}$ of $\mathbf G$. Further, because $\mathcal G^\pm(\mathbf C_{p^\infty})$ is a complete graph, and because $\overline N(\overline N(C))\subseteq\overline N(C)$, it follows that $\hat C=\overline N(C)$. Therefore, $\lvert\hat C\rvert-\lvert C\rvert =p^{s-1}$. This proves the lemma.
\end{proof}\\

Before heading over to prove the main theorem of this subsection, we have just one more proposition to prove. Proposition \ref{pomocno tvrdjenje za p-lupe na perin predlog} covers the case when there is a prime number $p$ such that the orders of all elements of finite order of the group are powers of $p$.

\begin{proposition}\label{pomocno tvrdjenje za p-lupe na perin predlog}
Let $p$ be a prime number, let $\mathbf G$ and $\mathbf H$ be groups in which orders of all elements of finite order are powers of $p$, and let $\mathbf G$ have no intersection-free quasicyclic subgroup. Let $\mathcal G^\pm(\mathbf G)$ and $\mathcal G^\pm(\mathbf H)$ have isomorphic finite-order components. Then $\vec{\mathcal G}^\pm(\mathbf G)$ and $\vec{\mathcal G}^\pm(\mathbf H)$ have isomorphic finite-order components too.
\end{proposition}

\begin{proof}
Since $\mathbf G$  contains no intersection-free quasicyclic subgroup, $\mathbf G$ is not a Pr\" ufer group, and so neither is $\mathbf H$. Furthermore, $\mathcal G^\pm(\mathbf G)$ has no infinitely complex $\equiv_{\mathbf G}$-class $C$ such that $\lvert\hat C\setminus C\rvert=1$. Therefore, $\mathcal G^\pm(\mathbf H)$ also does not contain such infinitely complex $\equiv_{\mathbf H}$-classes, which implies that $\mathbf H$ also does not contain any intersection-free quasicyclic subgroup.

Let us denote the finite-order components of $\mathcal G^\pm(\mathbf G)$ and $\mathcal G^\pm(\mathbf H)$ by $\Phi$ and $\Psi$, respectively. Let $\varphi:G_{<\infty}\rightarrow H_{<\infty}$ be an isomorphism from $\Phi$ onto $\Psi$. For a finite $\equiv$-class $C$ contained in $G_{<\infty}$ or $H_{<\infty}$, let us show that the maximal order of an element from $C$ is equal to $\lvert\hat C\rvert$. Let $c$ be an element of $C$ of maximal order. Then $\langle c\rangle\subseteq\hat C$. Suppose that there is an element $d\in\hat C\setminus \langle c\rangle$. Then the order of $d$ is greater than $o(c)$, and $\langle c\rangle<\langle d\rangle$, which implies $\overline N(d)\subseteq\overline N(c)$. Also, the fact that $d\in \overline N(\overline N(c))$ implies $\overline N(c)\subseteq\overline N(d)$, and thus $c\equiv d$, which is a contradiction. Therefore, the maximal order of an element of $C$ is $\lvert\hat C\rvert$. Also, for an infinitely complex $\equiv$-class $C$, $\lvert \hat C\rvert=\aleph_0=\sup\{o(x)\mid x\in C\}$.

Now, for $x,y\in G_{<\infty}$ such that $x\not\equiv_{\mathbf G} y$ and $x\simppm_{\mathbf G} y$, $x\rightarrowpm_{\mathbf G} y$ implies $\lvert\overline N_{\mathbf G}(\overline N_{\mathbf G}(y))\rvert <\lvert\overline N_{\mathbf G}(\overline N_{\mathbf G}(x))\rvert$. Because $\varphi$ is an isomorphism from $\Phi$ onto $\Psi$, $\lvert\overline N_{\mathbf H}(\overline N_{\mathbf H}(\varphi(y)))\rvert <\lvert\overline N_{\mathbf H}(\overline N_{\mathbf H}(\varphi(x)))\rvert$ and $\varphi(x)\simppm_{\mathbf H}\varphi(y)$, and thus, by Lemma \ref{svako svakog tuce}, $\varphi(x)\rightarrowpm_{\mathbf H}\varphi(y)$. Now it remains to prove that $\varphi$ determines the isomorphisms between pairs of $\equiv$-classes of $\mathcal G^\pm(\mathbf G)$ and $\mathcal G^\pm(\mathbf H)$.

Let $C$ be an $\equiv_{\mathbf G}$-class, and let $D=\varphi(C)$. Suppose that $C$ is finite. Then $D$ is a finite $\equiv_{\mathbf H}$-class too. If $C$ contains the identity element, then $C=\Cen(\Phi)$ and $D=\Cen(\Psi)$. In this case $\mathbf C$ and $\mathbf D$ are cyclic subgroups of the same order, which implies that $\big(\vec{\mathcal G}^\pm(\mathbf G)\big)[C]\cong\big(\vec{\mathcal G}^\pm(\mathbf H)\big)[\varphi(C)]$. So suppose that $C$ does not contain the identity element. Then there is an element $z$ from $\langle C\rangle\setminus C$ of maximal order. Let us denote $\lvert\langle C\rangle\rvert=\lvert\hat C\rvert$ by $p^r$, and let us denote $\lvert\langle z\rangle\rvert=\lvert\overline N_{\mathbf G}(\overline N_{\mathbf G}(z))\rvert$ by $p^{s-1}$. Then $C$ contains elements of orders $p^s,p^{s+1},\dots,p^r$. In the same way we conclude that $D$ also contains elements of orders $p^s,p^{s+1},\dots,p^r$. Also, notice that $\Phi[C]$ and $\Psi[D]$ are complete subgraphs. Therefore, $\big(\vec{\mathcal G}^\pm(\mathbf G)\big)[C]\cong\big(\vec{\mathcal G}^\pm(\mathbf H)\big)[\varphi(C)]$ for any finite $\equiv_{\mathbf G}$-class $C$ of $\Phi$. 

Suppose now that $C$ is an infinitely-complex $\equiv_{\mathbf G}$-class. Then $\lvert\hat C\setminus C\rvert=p^{s-1}$ for some $s>1$. Because $\varphi$ is an isomorphism from $\mathcal G^\pm(\mathbf G)$ onto $\mathcal G^\pm(\mathbf H)$, $\lvert\hat D\setminus D\rvert=p^{s-1}$. It follows that both $C$ and $D$ are infinitely complex $\equiv$-classes containing elements of orders $p^s,p^{s+1},p^{s+2},\dots$, which implies $\big(\vec{\mathcal G}^\pm(\mathbf G)\big)[C]\cong\big(\vec{\mathcal G}^\pm(\mathbf H)\big)[\varphi(C)]$. This proves that $\vec{\mathcal G}^\pm(\mathbf G)$ and $\vec{\mathcal G}^\pm(\mathbf H)$ have isomorphic finite-order components.
\end{proof}

\begin{theorem}\label{neusmereni odredjuju usmerene kod konacnih}
Let $\mathbf G$ and $\mathbf H$ be groups, and let $\mathbf G$ have no intersection-free quasicyclic subgroup. If $Z^\pm$-power graphs of $\mathbf G$ and $\mathbf H$ have isomorphic finite-order components, then their directed $Z^\pm$-power graphs have isomorphic finite-order components too.
\end{theorem}

\begin{proof}
Let $\mathbf G$ and $\mathbf H$ be groups, and let $\mathbf G$ have no intersection-free quasicyclic subgroup. Let us denote the finite-order components of $\mathcal G^\pm(\mathbf G)$ and $\mathcal G^\pm(\mathbf H)$ by $\Phi$ and $\Psi$, respectively. Let $\Phi\cong\Psi$, and let $\lvert\Cen(\Phi)\rvert>1$. Then $\lvert\Cen(\Psi)\rvert>1$ too. Now, by Proposition \ref{kameronova propozicija}, $\mathbf G_{<\infty}$ and $\mathbf H_{<\infty}$ are either both cyclic groups of the same order, or there is a prime number $p$ such that the order of every element of $G_{<\infty}$ and $H_{<\infty}$ is a power of $p$. In the first case, trivially, the $Z^\pm$-power graphs  of $\mathbf G$ and $\mathbf H$ have isomorphic finite-order components, while in the second case that follows by Proposition \ref{pomocno tvrdjenje za p-lupe na perin predlog}.

Suppose further that $\Phi\cong\Psi$ and $\lvert\Cen(\Phi)\rvert=1$. Then $\lvert\Cen(\Psi)\rvert=1$ too. Let $\psi:G\rightarrow H$ be an isomorphism from $\Phi$ onto $\Psi$. Notice that if $C$ is an $\equiv_{\mathbf G}$-class, then $\psi(C)$ is also an $\equiv_{\mathbf H}$-class. Also, by Lemma \ref{prepoznavanje tipova klasa}, $C$ and $\psi(C)$ are either both simple $\equiv$-classes, or they are both complex $\equiv$-classes, or they are both infinitely-complex $\equiv$-classes. For sets $X\subseteq G_{<\infty}$ and $Y\subseteq H_{<\infty}$, we say that they are corresponding if there is an $\equiv_{\mathbf G}$-class $C$ such that $X\subseteq C$ and $Y\subseteq\psi(C)$. 

Just like in the proof of Proposition \ref{pomocno tvrdjenje za p-lupe na perin predlog}, from the fact that $\mathbf G$ does not have any intersection-free quasicyclic subgroup, we conclude that $\mathbf H$ has no intersection-free quasicyclic subgroup too. Therefore, by Lemma \ref{prepoznavanje redova u beskonacno komplesknoj klasi}, for any infinitely complex $\equiv$-class of $\Phi$ or $\Psi$ it is possible to determine all orders of elements contained in $C$. Also, by Lemma \ref{prepoznavanje tipova klasa}, for every complex $\equiv$-class contained in $G_{<\infty}$ or $H_{<\infty}$ one can determine the orders of elements contained in this $\equiv$-class. Therefore, each complex or infinitely complex $\equiv_{\mathbf G}$-class contains elements of the same orders as its corresponding $\equiv_{\mathbf H}$-class. In the remainder of this proof, when mentioning an $\approx$-class or an $\equiv$-class, we assume that those are classes that contain elements of finite order. We conclude that there is a bijection $\vartheta: G_{<\infty}\rightarrow H_{<\infty}$ which maps every $\approx_{\mathbf G}$-class onto the corresponding $\approx_{\mathbf H}$-class. Notice that, by Lemma \ref{svako svakog tuce}, all $\approx$-classes contained in the same complex or infinitely complex $\equiv$-class $C$ relate in the same way in the directed $Z^\pm$-power graph to any other $\approx$-class outside the $\equiv$-class $C$. Finally, $\vartheta$ is an isomorphism from $\big(\vec{\mathcal G}(\mathbf G)\big)[G_{<\infty}]$ onto $\big(\vec{\mathcal G}(\mathbf H)\big)[H_{<\infty}]$, because, by Lemma \ref{odredjivanje pravaca}, for two adjacent $\approx$-classes one can tell which one contains elements of greater order. This proves the theorem.
\end{proof}

\subsection{Putting the Pieces Together}

Now we are ready to prove the main result of this paper.

\begin{theorem}\label{Zpm-stepeni graf skoro uvek odredjuje usmereni stepeni}
Let $\mathbf G$ and $\mathbf H$ be groups such that $\mathcal G^\pm(\mathbf G)\cong\mathcal G^\pm(\mathbf H)$. If $\mathbf G$ has no intersection-free quasicyclic subgroup, then $\vec{\mathcal G}^\pm(\mathbf G)\cong\vec{\mathcal G}^\pm(\mathbf H)$.
\end{theorem}

\begin{proof}
By Lemma \ref{konacni se slikaju na konacne}, $\mathcal G^\pm(\mathbf G)$ and $\mathcal G^\pm(\mathbf H)$ have isomorphic infinite-order components and isomorphic finite-order components. Then, by Theorem \ref{izomorfni podgrafovi sa elementima beskonacnog reda} and Theorem \ref{neusmereni odredjuju usmerene kod konacnih}, $\vec{\mathcal G}^\pm(\mathbf G)$ and $\vec{\mathcal G}^\pm(\mathbf H)$, too, have isomorphic infinite-order components and isomorphic finite-order components. Therefore, groups $\mathbf G$ and $\mathbf H$ have isomorphic directed $Z^\pm$-power graphs.
\end{proof}\\

By Theorem \ref{medjusobno odredjivanje opsteg i nenula-stepenog grafa}, the ensuing corollary follows directly from Theorem \ref{Zpm-stepeni graf skoro uvek odredjuje usmereni stepeni}.

\begin{corollary}\label{stepeni graf skoro uvek odredjuje usmereni stepeni}
Let $\mathbf G$ and $\mathbf H$ be groups such that $\mathcal G(\mathbf G)\cong\mathcal G(\mathbf H)$. If $\mathbf G$ has no intersection-free quasicyclic subgroup, then $\vec{\mathcal G}(\mathbf G)\cong\vec{\mathcal G}(\mathbf H)$.
\end{corollary}


\section*{Acknowledgment}

The author acknowledges financial support of the Ministry of Education, Science and Technological Development of the Republic of Serbia (Grant No. 451-03-68/2020-14/200125).


\end{document}